\documentclass[reqno]{amsart}

\usepackage{xypic}
\input xy
\xyoption{all}
\usepackage{epsfig}
\usepackage{amsthm}
\usepackage{amssymb}
\usepackage{amsmath}
\usepackage{amscd}
\usepackage{color}
\usepackage{enumitem}

\newcommand{\lb}{\varLambda}
\newcommand{\rf}[1]{(\ref{#1})}

\def \<{\langle}
\def \>{\rangle}

\newcommand{\bg}{\begin{equation}}
\newcommand{\ed}{\end{equation}}
\newcommand{\bga}{\begin{eqnarray}}
\newcommand{\eda}{\end{eqnarray}}

\def\cbdu{\par{\raggedleft$\Box$\par}}

\newtheorem {Theorem}  {Theorem}
\newtheorem {Corollary}[Theorem]{\bf Corollary}

\numberwithin{Theorem}{section}

\newtheorem {Lemma}[Theorem]  {Lemma}

\theoremstyle{definition}
\newtheorem{Definition}[Theorem]{Definition}
\theoremstyle{remark}
\newtheorem{Remark}[Theorem]{\bf Remark}

\def \l{\lambda}
\expandafter\chardef\csname pre amssym.def
at\endcsname=\the\catcode`\@ \catcode`\@=11
\def\undefine#1{\let#1\undefined}
\def\newsymbol#1#2#3#4#5{\let\next@\relax
 \ifnum#2=\@ne\let\next@\msafam@\else
 \ifnum#2=\tw@\let\next@\msbfam@\fi\fi
 \mathchardef#1="#3\next@#4#5}
\def\mathhexbox@#1#2#3{\relax
 \ifmmode\mathpalette{}{\m@th\mathchar"#1#2#3}%
 \else\leavevmode\hbox{$\m@th\mathchar"#1#2#3$}\fi}
\def\hexnumber@#1{\ifcase#1 0\or 1\or 2\or 3\or 4\or 5\or 6\or 7\or 8\or
 9\or A\or B\or C\or D\or E\or F\fi}

\font\teneufm=eufm10 \font\seveneufm=eufm7 \font\fiveeufm=eufm5
\newfam\eufmfam
\textfont\eufmfam=\teneufm \scriptfont\eufmfam=\seveneufm
\scriptscriptfont\eufmfam=\fiveeufm

\catcode`\@=\csname pre amssym.def at\endcsname

\newcounter{remark}
\newenvironment{remark}
{\medskip \stepcounter{remark} \noindent \textit{Remark
\arabic{section}.\arabic{remark}.}}{\rm \cbdu}

\def \grad {\nabla}
\def \ls {\lambda_q^{2s}}

\newcommand{\el}{\varepsilon}

\newcommand{\divv}{{\mathit div}\,}

\renewcommand{\d}{\delta}
\newcommand{\om}{\omega}
\renewcommand{\k}{\kappa}
\renewcommand{\l}{\lambda}

\renewcommand{\a}{\alpha}

\newcommand{\s}{\sigma}

\newcommand{\R}{\mathbf{R}}

\newcommand{\Ss}{{\mathcal S}}

\def  \R   {{\mathbb R}}
\def  \Z   {{\mathbb Z}}
\def  \N   {{\mathbb N}}

\def  \T   {{\mathbb T}}

\def  \haf  {{\frac{1}{2}}}
\def  \12  {{\frac{1}{2}}}
\def  \p   {\partial}

\def\build#1_#2^#3{\mathrel{\mathop{\kern 0pt#1}\limits_{#2}^{#3}}}

\begin{document}

\title[Determining Wavenumber for 3D NSE]{An optimal upper bound on the determining wavenumber for 3D Navier-Stokes Equations}

\author [Alexey Cheskidov]{Alexey Cheskidov}
\address{Institute for Theoretical Sciences, Westlake University, Hangzhou, China}
\email{cheskidov@westlake.edu.cn} 
\author [Qirui Peng]{Qirui Peng}
\address{Department of Mathematics, Stat. and Comp. Sci.,  University of Illinois Chicago, Chicago, IL 60607,USA}
\email{qpeng9@uic.edu}






\begin{abstract} We introduce a determining wavenumber for weak solutions of 3D Navier-Stokes equations whose time average is bounded by Kolmogorov dissipation wavenumber over the whole range of intermittency dimensions. This improves previous works \cite{CDK} and \cite{CD1}.

\bigskip


\hspace{0.02cm}CLASSIFICATION CODE: 35Q35, 37L30.\\

\end{abstract}

\maketitle

\section{Introduction}

The Navier-Stokes equations (NSE) on a torus $\mathbb{T}^3=[0,L]^3$  are given by 
\begin{equation}
  \label{nse}
  \left\{
    \begin{array}{l}
      u_t + (u \cdot \grad) u - \nu\Delta u + \grad p = f \\
      \grad \cdot u = 0,
    \end{array}
  \right.
\end{equation}
where $u$ stands for the velocity, $p$ the pressure, $f$ the external force and $L > 0$ the size of the torus. Here $u$ and $f$ are assumed to be zero mean. 


In turbulent flows, the energy injected at forced low modes (large scales) cascades to small scales through the inertial range where viscous effects are negligible, and only dissipates above a certain mode, known as Kolmogorov’s dissipation wavenumber $\kappa_{\mathrm{d}}$. As predicted by Landau, the number of degrees of freedom of a three-dimensional turbulent fluid flow is expected to be of the order $\kappa_{\mathrm{d}}^3$. The aim of this note is to justify this prediction rigorously.

To define the number of degrees of freedom rigorously, the notion of determining modes was introduced by Foias and Prodi in \cite{FP} where they showed that high modes of
a solution to the 2D NSE are controlled by low modes asymptotically as time goes to infinity. In \cite{FMTT}, Foias, Manley, Temam, and Treve proveided an estimate for the number of determining modes. This was improved  later by Jones and Titi  \cite{JT}.
For more background, we refer the readers to  \cite{CFMT,CFT,FJKT,FJKTgeneral,FMRT,FT,FT84, FT-attractor,FTiti}.

For the 3D NSE the situation is more delicate as the global existence of regular solutions is not known. Nevertheless, the notion of determining modes can be generalized in the following sense: 
\begin{Definition}\label{def_determining_wavenumbers}
We call a family of functions $\Lambda_u:[0,\infty) \to \mathbb{R}_+$, assigned to each weak solution $u(t)$ of the 3D NSE, determining wavenumbers, if 
\[
P_{\leq \Lambda_u(t)}\big(u(t)-u(t)\big) = P_{\leq \Lambda_v(t)}\big(v(t)-v(t)\big)=0, \qquad \forall t >0
\]
implies
\[
u(t) -v(t) \to 0, \quad \text{as} \quad t \to \infty, 
\]
weakly in $L^2$.
\end{Definition}
Here $P_{\leq \Lambda}$ is a projection on frequencies below $\Lambda$. As was proved in \cite{CDK,CD1}, there exists a family of determining wavenumbers, such that for every weak solution satisfying the energy inequality starting from $t=0$,  an upper bound on time averages of $\Lambda_u$ holds for $t$ large enough (i.e., in an absorbing ball),
\[
\<\Lambda_u\>(t) :=  \frac{1}{T}\int_t^{t+T} \Lambda_u(\tau) \, d\tau \lesssim C_{\nu,f},
\]
for some constant $C_{\nu,f}$ that only depends on the viscosity coefficient $\nu$ and size of the force $f$, but not on the initial data (or a choice of the weak solution). The same bound folds for the time averages of $\Lambda_u(t)$ on the pullback global attractor. This indicates that the average number of degrees of freedom of the fluid flow is bounded by $C_{\nu,f}^{3}$.

To make a connection with Landau's prediction, it is desirable to compare $\<\Lambda_u\>$ with Kolmogorov's dissipation wave number defined as
\[
\kappa_\mathrm{d} := \left(\frac{\varepsilon }{\nu^3} \right)^{\frac{1}{d+1}}, 
\]
where $\varepsilon = \nu\l_0^d\<\|\nabla u\|_2^2\>$ is the average energy dissipation and $d\in[0,3]$ is the average intermittency dimension of the fluid flow, see Section~\ref{sec:Intermittency}.

In \cite{CDK}, the author defined a determining wavenumber with an upper bound which was physically optimal $ \<\lb \> - \frac{1}{L} \lesssim \kappa_d$ only in the case of extreme intermittency $d=0$ not observed in real turbulent flows. In a subsequent work 
\cite{CD1}, another definition of a determining wavenumeber was introduces, but only for trajectories with intermittency dimension $d \in [\d,3]$ for some a priori fixed $\d > 0$. The bound was physically optimal in this regime, but slightly deteriorated as $d \to 3$ and produced a logarithmic correction in Kolmogorov's regime $d=3$.

In this paper we show the existence of a family of determining number enjoying the following upper bound
\[
 \<\lb \> - \frac{1}{L} \lesssim \kappa_d \min \Big \{(\ln{(L\kappa_d)})^2, \frac{1}{d^2}\Big \}.
\]
Hence, for fluid flows with intermittency dimension $d$ above some absolute positive constant, the physically optimal bound is achieved 
\[
 \<\lb \> - \frac{1}{L} \lesssim \kappa_d.
\]
This is a physically relevant regime, as the intermittency dimension is observed to be close to $3$ for turbulent fluid flows.
Furthermore, we show that the above optimal bound holds for a very small intermittency dimension $d \in \big[ 0, (\ln{(L\kappa_d)})^{-1}\big] $ as well.

We define the determining wavenumber for the solution $u$ of the 3D NSE as follows: Firstly, denote the set $\Ss$ by
\begin{equation}\label{domain_of_index}
\Ss := \Big \{ (r,\d) \Big| \ r \in [2,\infty], \   \d \in [0,3],  \text{ and } (r,\d) \ne (\infty, 0)  \Big \}
\end{equation}
which stands for the domain of the tuple $(r,\d)$. Next, we introduce a wavenumber for each tuple $(r,\d) \in \Ss$
\begin{equation}\label{def_det_wave}
\begin{split}
\lb^{r,\delta}_{u}=\min\{\lambda_q: (L\lambda_{p-q})^{\frac3{r}-\frac12+\frac{\delta}{2}}\lambda_q^{-1+\frac3r}\|u_p\|_r<c_{r,\delta} \nu, \forall p>q, \\
\text{and} \ \ \lambda_q^{-2}\|\nabla u_{\leq q}\|_\infty<c_{r,
\delta}\nu,  q\in \N.\}
\end{split}
\end{equation}
where $c_{r,\delta}$ is an adimensional constant that depends only on $r$ and $\delta$. In fact, $c_{r,\delta} \to 0$ as $ (r,\delta) \to (\infty,0)$. The explicit form of $c_{r,\d}$ will be given by the end of Section $\ref{sec:pf}$. Additionally, $\lambda_q =\frac{2^q}{L}$, $L$ is the size of the torus, $u_{\leq q}= \sum_{p=-1}^q u_q$, and $u_q = \Delta_q u$ is the Littlewood-Paley projection of $u$ (see Section~\ref{sec:pre}). Note that a convention $\min \emptyset = \infty$ is adopted in the
definition of $\lb^{r,\d}_{u}(t)$. Finally, we choose $\lb_u (t)$ as the smallest wavenumbers among the family of the wavenumbers $\lb^{r,\d}_u(t)$, i.e.
\begin{equation}\label{def_uni_det_wave}
\lb_{u}(t):=\min_{(r,\d) \in \Ss}\lb^{r,\delta}_u(t)
\end{equation}

Now we are ready to state our main result, which states that $\lb_u(t)$ is a determining wavenumber. In fact, the proof shows that 

\begin{Theorem}\label{main_Theorem}
Let $u(t)$ and $v(t)$ be two weak solutions of the 3D Navier-Stokes equations.
Let $\lb(t):=\max\{\lb_{u}(t), \lb_{v}(t)\}$ and $Q(t)$ be such that $\lb(t)=\lambda_{Q(t)}$.
If
\begin{equation} \label{eq:dm-condition}
u(t)_{\leq Q(t)}=v(t)_{\leq Q(t)}, \qquad \forall t>0,
\end{equation}
then
\[
\lim_{t\to \infty} ||u(t)-v(t)||_{H^{-\frac{5}{4}}} = 0.
\]   
\end{Theorem}
\begin{remark}\label{rmk_tuple_value}
    Without loss of generality let $\lb(t) = \lb_u(t)$. Note that from the definition $(\ref{def_uni_det_wave})$ there exists $(r,\d)\in \Ss$ such that $\lb_u(t) = \lb^{r,\d}_u(t)$. The convergence in Theorem $\ref{main_Theorem}$ in fact holds for the $H^s$ norm with
    \begin{equation}\label{def_value_s}
        s := \min \Big \{-\frac{\d}{2} + \frac{1}{4}, 0 \Big \} \geq - \frac{5}{4}.
    \end{equation}
\end{remark}

\begin{Corollary} \label{main_Corollary}
    Let $u(t)$ and $v(t)$ as in Theorem $\ref{main_Theorem}$. If in addition, $u(t)$ and $v(t)$ are both Leray-Hopf solutions, then we have that 
    \[
   \lim_{t \to \infty} \<\|u - v \|^2_{H^{s}} \>(t) = 0
    \]
    for any $s < 1$.
\end{Corollary}

\section{Preliminaries}
\label{sec:pre}
In this section, we introduce several technical tools and notations for convenience of further presentation.

\subsection{Notation}

\label{sec:notation}
By $A\lesssim B$ we mean for some constant $C > 0$, we have that $A \leq C B$. The term $A \sim B$ refers to the bound $c_1 B \leq A \leq C_1 B$ for some positive constant $c_1, C_1$. In addition, the term $A \lesssim_s B$ means there exist an adimensional constant $C_s > 0$ that is only dependent on $s$ such that $A \leq C_s B$. We use the notation $\|\cdot\|_p=\|\cdot\|_{L^p}$ for $L^p$ norm and $(\cdot, \cdot)$ for the standard $L^2$-inner product. In this paper, the time averages are represented by the following:
\[
\<g\>(t) := \frac{1}{T}\int_t^{t+T} g(\tau) \, d\tau, 
\]
for some $T>0$.

\subsection{Littlewood-Paley theory}
\label{sec:LPD}
We recall here the standard Littlewood-Paley theory. For more detail, we refer to the books by Bahouri, Chemin and Danchin \cite{BCD} and Grafakos \cite{Gr}.  

Consider the physical domain $\T^3 = [0,L]^3$. For any integer $q \in \Z$, we denote the number $\lambda_q :=\frac{2^q}{L}$. Let $\chi\in C_0^\infty(\R^3)$ be the non-negative radial function such that 
\begin{equation} \label{eq:xi}
\chi(\xi):=
\begin{cases}
1, \ \ \mbox { for } |\xi|\leq\frac{3}{4}\\
0, \ \ \mbox { for } |\xi|\geq 1.
\end{cases}
\end{equation}
We then define
\[
\varphi(\xi):=\chi(\xi/2)-\chi(\xi)
\]
as well as
\begin{equation}\notag
\varphi_q(\xi):=
\begin{cases}
\varphi(2^{-q}\xi)  \ \ \ \mbox { for } q\geq 0,\\
\chi(\xi) \ \ \ \mbox { for } q=-1,
\end{cases}
\end{equation}
the sequence $\varphi_q$ then forms a dyadic partition of unity. Let $u$ be a tempered distribution vector field on $\T^3 =[0,L]^3$ and $q$ an integer larger than $-1$, then the $q^{\text{th}}$ Littlewood-Paley projection of $u$ is given by 
\[
 \Delta_q u(x) := \sum_{k\in\Z^3}\hat{u}(k)\phi_q(k)e^{i\frac{2\pi}{L} k \cdot x},
\]
where $\hat{u}(k)$ is the $k$th Fourier coefficient of $u$. For the clarity of presentation, we sometimes also use $u_q(x)$ as the Littlewood-Paley projection: 
\[
u_q (x) := \Delta_q u(x)
\]

It follows that 
\[
u=\sum_{q=-1}^\infty u_q
\]
in the distributional sense. The $H^s$-norm of $u$ can be defined in the following way:
\begin{equation}\label{def_Hs_norm}
    \|u\|_{H^s} := \left(\sum_{q=-1}^\infty\lambda_q^{2s}\|u_q\|_2^2\right)^{1/2},
\end{equation}
for each $u \in H^s$ and $s\in\R$. Note that $\|u\|_{H^0} \sim \|u\|_{L^2}$. Finally we denote 
\bg\notag
u_{\leq Q}:=\sum_{q=-1}^Q u_q, \quad u_{(p,Q]}:=\sum_{q=p+1}^Qu_q, \quad \tilde{u}_q := u_{q-1} + u_q + u_{q+1}.
\ed
for simplicity of notation.

\subsection{Technical lemmas}
\label{sec-para}
We recall several important lemmas which will be very useful in the proof of Theorem $\ref{main_Theorem}$. 

\begin{Lemma}\label{le:bern}(Bernstein's inequality) 
Let $n$ be the spacial dimension and $r\geq s \geq 1$, then for tempered distribution $u$ it follows that
\bg\label{Bern}
\|u_q\|_{r}\lesssim \lambda_q^{n(\frac{1}{s}-\frac{1}{r})}\|u_q\|_{s}.
\ed
\end{Lemma}

\begin{Lemma} \label{lemma:Bony} (Bony's paraproduct formula) Let $u$ and $v$ be two tempered distributions, then the following holds: 
\begin{equation}\notag
\begin{split}
\Delta_q(u\cdot\nabla v)=&\sum_{|q-p|\leq 2}\Delta_q(u_{\leq{p-2}}\cdot\nabla v_p)+
\sum_{|q-p|\leq 2}\Delta_q(u_{p}\cdot\nabla v_{\leq{p-2}})\\
&+\sum_{p\geq q-2} \Delta_q(\tilde u_p \cdot\nabla v_p).
\end{split}
\end{equation}
\end{Lemma}

\begin{Lemma} \label{lemma: Jensen} (Jensen's inequality) Let $\phi : I \to \R$ be a convex function on $I \subset \R$. For $x_1, x_2, ..., x_n \in I$ and $\a_1, \a_2, ..., \a_n \in (0,1)$ such that 
\bg \notag
\sum_{j=1}^n \a_j = 1
\ed
it follows that 
\begin{equation}\notag
\phi \Big( \sum_{j=1}^n \a_j x_j \Big) \leq \sum_{j=1}^n \a_j \phi (x_j).
\end{equation}
\end{Lemma}

\subsection{Weak solutions and energy inequality} 
We call an $L^2(\T^3)$ valued function $u(t)$ a weak solution of \eqref{nse} on $[0,\infty)$ if $u \in C([0,\infty); L^2_{\mathrm{w}}) \cap L_{\mathrm{loc}}^2(0,\infty; H^1)$ solves \eqref{nse} in the sense of distributions. A Leray-Hopf solution $u(t)$ is a weak solution satisfying the energy inequality 
\begin{equation} \label{eq:EI-sec3}
\frac{1}{2}\|u(t)\|_2^2 \leq \frac{1}{2}\|u(t_0)\|_2^2 - \nu\int_{t_0}^{t} \|\nabla u(\tau)\|_2^2\, d\tau + \int_{t_0}^{t} (f,u)\, d\tau,
\end{equation}
for almost everywhere $t_0 > 0$ and all $t >t_0$. A Leray solution $u(t)$ is a Leray-Hopf solution satisfying the above energy inequality for $t_0=0$ and all $t>t_0$.

\bigskip

\section{The Intermittency dimension} \label{sec:Intermittency}

We start from the definition of the Kolmogorov's dissipation wavenumber:
\begin{equation} \label{eq:kdeps-inermit}
\kappa_\mathrm{d} := \left(\frac{\varepsilon }{\nu^3} \right)^{\frac{1}{d+1}}, \qquad  \varepsilon := \nu\l_0^d\<\|\nabla u\|_2^2\>,
\end{equation}
here $d$ stands for the intermittency dimension and $\el$ the average energy dissipation rate per unit active volume. Observe that by Bernstein's inequality, we have for $r \geq 2$,
\begin{equation} \label{eq:BinInt}
\l_0^{3-\frac{6}{r}}\l_q^{-1+\frac{6}{r}} \|u_q\|_2^2 \lesssim \l_q^{-1+\frac{6}{r}} \|u_q\|_r^2 \lesssim \l_q^{2} \|u_q\|_2^2,
\end{equation}
The intermittency dimension $d$ describes the level of saturation of Bernstein's inequality (see also \cite{CSr,CSint,CSint-new} for similar definitions) and is defined by
\begin{equation} \label{eq:intermdef}
d:= \sup\left\{D\in [0,3]: r \in [2,\infty], \left<\sum_{q\leq Q(t)}\l_q^{-1+\frac{6}{r}+D(1-\frac{2}{r})} \|u_q\|_r^2 \right> \lesssim  \l_0^{D(1-\frac{2}{r})}\left<\sum_{q \leq Q(t)}\l_q^{2} \|u_q\|_2^2 \right> \right\},
\end{equation}
for $u \not\equiv 0$, and $d=3$ for $u \equiv 0$ on $[t, t+T]$. Recall here that $Q(t)$ is such that $\Lambda(t) = \l_{Q(t)}$.
Thanks to \eqref{eq:BinInt} and the fact that $\<\sum_{q}\l_q^{2} \|u_q\|_2^2 \><\infty$ , we have that
\begin{equation}\label{interm_estimate}
\left<\sum_{q\leq Q(t)}\l_q^{-1+\frac{6}{r}+d(1-\frac{2}{r})} \|u_q\|_r^2 \right> \lesssim  \l_0^{d}\left<\sum_{q \leq Q(t)}\l_q^{2} \|u_q\|_2^2 \right>.
\end{equation}
for any $r \in [2,\infty]$.

\section{Proof of the main result}
\label{sec:pf}

\begin{proof}[Proof of Theorem \ref{main_Theorem}]

Let $(u(t), p_1(t))$ and $(v(t),p_2(t))$ be two weak solutions of NSE. Denote their difference by $w:=u-v$, then it satisfies the equation
\begin{equation} \label{eq-w}
w_t+u\cdot\nabla w+w\cdot\nabla v=-\nabla p'+\nu \Delta w
\end{equation}
in the sense of  distributions, and here $p'=p_1 -p_2$. For a more convenient notation let us also denote that
\begin{equation}\label{def_sigma}
   \s := \haf \big(\d-1 \big).
\end{equation}
Since $\delta \in [0,3]$, we have $-\haf \leq \sigma \leq 1$. Recall from the definition $\rf{def_det_wave}$ of the wavenumber: 
\begin{equation*}
\begin{split}
\lb^{r,\delta}_{u}=\min\{\lambda_q: (L\lambda_{p-q})^{\frac3{r}+\sigma }\lambda_q^{-1+\frac3r}\|u_p\|_r<c_{r,\delta} \nu, \forall p>q, \\
\text{and} \ \ \lambda_{q}^{-2}\|\nabla u_{\leq q}\|_\infty<c_{r,
\delta}\nu, q\in \N \}
\end{split}
\end{equation*}
and from $\rf{def_uni_det_wave}$ that
\begin{equation*}
\lb_{u}(t):=\min_{(r,\d)\in \Ss}\lb^{r,\delta}_u(t)
\end{equation*} 
Let $\lb(t):=\max\{\lb_{u}(t), \lb_{v}(t)\}$ and $Q(t)$ be such that $\lb(t)=\lambda_{Q(t)}$, hence by assumption, $w_{\leq Q(t)}(t)\equiv 0$.
Take the tuple $(r,\d) \in \Ss$ such that $\lb^{r,\d}_u(t) = \lb(t)$ (we can without loss of generality assume that $\lb(t) = \lb_u(t)$). In addition, denote
\begin{equation}\label{def_s}
s =\min\left\{\textstyle -\frac12+\frac{\delta}{4},0\right\}.
\end{equation}
A simple computation yields 
\begin{equation}\label{s_value_range}
-1-\sigma \leq s \leq \sigma\leq 1.
\end{equation}

Multiplying the equation (\ref{eq-w}) with $\ls\Delta^2_qw$, integrating over space and time and summing over $q \geq -1$ gives
\begin{equation}\label{eq-w2}
\begin{split}
\frac{1}{2}\|w(t)\|_{H^s}^2- \frac{1}{2}\|w(t_0)\|_{H^s}^2 &+\nu \int_{t_0}^t\|w\|_{H^{1+s}}^2 \, d\tau\\
&\leq \int_{t_0}^t \sum_{q\geq -1}\ls\left|\int_{\T^3}\Delta_q(w\cdot\nabla v) w_q\, dx\right|\, d\tau\\
&+\int_{t_0}^t\sum_{q\geq -1}\ls\left|\int_{\T^3}\Delta_q(u\cdot\nabla w) w_q\, dx\right|\, d\tau,\\
&=:\int_{t_0}^t N \, d\tau+ \int_{t_0}^t M\, d\tau.
\end{split}
\end{equation}

The term $M$ can be decomposed via Bony's paraproduct formula (see lemma $\ref{lemma:Bony}$):
\begin{equation}\notag
\begin{split}
M\leq& \sum_{q\geq -1}\ls\sum_{|q-p|\leq 2}\left|\int_{\T^3}\Delta_q(w_{\leq{p-2}}\cdot\nabla v_p) w_q\, dx\right|\\
&+\sum_{q\geq -1}\ls\sum_{|q-p|\leq 2}\left|\int_{\T^3}\Delta_q(w_{p}\cdot\nabla v_{\leq{p-2}})w_q\, dx\right|\\
&+\sum_{q\geq -1}\ls\sum_{p\geq q-2} \left|\int_{\T^3}\Delta_q(\tilde w_p\cdot\nabla v_p)w_q\, dx\right|\\
=:&M_{1}+M_{2}+M_{3}.
\end{split}
\end{equation}
For $M_1$, by H\"older's inequality,
\[
\begin{split}
  M_{1} 
& \leq \sum_{q>Q}\sum_{\substack{|q-p|\leq 2\\p>Q+2}}\ls\int_{\T^3}|\Delta_q(w_{\leq{p-2}}\cdot\nabla v_p) w_q|\, dx\\
 &\lesssim  \sum_{q>Q}\sum_{\substack{|q-p|\leq 2\\p>Q+2}}\ls\|w_{(Q, p-2]}\|_{\frac{2r}{r-2}}\lambda_p\|v_p\|_r\|w_q\|_2.
\end{split}
\]
Recall that by definition $\lb(t) = \l_{Q(t)}$, we obtain
\[
\begin{split}
 M_{1}  &\lesssim  c_{r,\delta}\nu\sum_{q>Q}\sum_{\substack{|q-p|\leq 2\\p>Q+2}}\ls\lb^{1+\sigma}\lambda_p^{1-\sigma-\frac{3}{r}}\|w_q\|_2\sum_{Q<p'\leq p-2}\|w_{p'}\|_{\frac{2r}{r-2}} \\
 &\lesssim  c_{r,\delta}\nu\sum_{q>Q}\lambda_q^{1+s}\|w_q\|_2\left(\sum_{Q<p'\leq q}\lambda_{p'}^{1+s}\|w_{p'}\|_2\lambda_{p'}^{-1-s+\frac{3}{r}}\lambda_{q}^{s-\sigma-\frac{3}{r}}\lambda_Q^{1+\sigma}\right)\\
&\lesssim  c_{r,\delta}\nu \sum_{q>Q}\lambda_q^{1+s}\|w_q\|_2\left(\sum_{Q<p'\leq q}\lambda_{p'}^{1+s}\|w_{p'}\|_2(L\l_{q-p'})^{s-\sigma-\frac{3}{r}}\right),
\end{split}
\]
where we also use the fact that $\sigma\geq -\haf$ and $s \leq \sigma$.
From the Young's inequality and Jensen's inequality,
\[
\begin{split}
M_1&\lesssim  c_{r,\delta}\nu \sum_{q>Q}\lambda_q^{2+2s}\|w_q\|_2^2+c_{r,\delta}\nu \sum_{q>Q}\left(\sum_{Q<p'\leq q}\lambda_{p'}^{1+s}\|w_{p'}\|_2(L\l_{q-p'})^{s-\sigma-\frac{3}{r}}\right)^2\\
&\lesssim c_{r,\delta}\nu \sum_{q>Q}\lambda_q^{2+2s}\|w_q\|_2^2+c_{r,\delta}\nu \sum_{q>Q}\left (\sum_{Q<p'\leq q}(L\l_{q-p'})^{s-\sigma-\frac{3}{r}}\right )\sum_{Q<p'\leq q}\lambda_{p'}^{2+2s}\|w_{p'}\|_2^2(L\l_{q-p'})^{s-\sigma-\frac{3}{r}}\\
&\lesssim  c_{r,\delta}\nu \sum_{q>Q}\lambda_q^{2+2s}\|w_q\|_2^2+c_{r,\delta}\nu \sum_{p'>Q}\lambda_{p'}^{2+2s}\|w_{p'}\|_2^2\left (\sum_{q\geq p'}(L\l_{q-p'})^{s-\sigma-\frac{3}{r}}\right )^2\\
&\lesssim_{r,\d}  c_{r,\delta}\nu \|w\|_{H^{s+1}}^2,
\end{split}
\]
the upper bound is applicable since $s - \sigma - \frac{3}{r} < 0$. To be precise, the bound for $M_1$ is
\bg \label{precise_bound1}
M_{1} \lesssim c_{r,\delta}\nu \|\nabla^{1+s} w\|_2^2\left(1+(1-2^{s-\sigma-\frac{3}{r}})^{-2}\right).
\ed
Since $(1-2^{s-\sigma-\frac{3}{r}})^{-1} \to \infty$ as $\delta \to 0^+$ and $r \to \infty$ by definitions of $\sigma$ and $s$, the choice of 
$c_{r,\delta}$ is such that $c_{r,\delta}\to 0$ as $\delta\to 0+$ and $r \to \infty$, which we will see by the end of the proof. Note that from $\rf{domain_of_index}$ we are free from the choice of $(r,\d) = (\infty,0)$.

Next, we estimate $M_2$ with the similar manner,
\[
\begin{split}
  M_{2} &\leq \sum_{q>Q}\sum_{\substack{|q-p|\leq 2\\p>Q}}\ls\int_{\T^3}|\Delta_q(w_p\cdot\nabla v_{\leq p-2}) w_q| \, dx\\
&\leq  \sum_{q>Q}\sum_{\substack{|q-p|\leq 2\\p>Q}}\ls\|w_p\|_2\|\nabla v_{(Q,p-2]}\|_r\|w_q\|_\frac{2r}{r-2}\\
&+\sum_{q>Q}\sum_{\substack{|q-p|\leq 2\\p>Q}}\ls\|w_p\|_2\|\nabla v_{\leq Q}\|_\infty\|w_q\|_2\\
&\equiv M_{21}+M_{22},
\end{split}
\]
and here if $p-2 \leq Q$, we use the convention that $(Q,p-2] = \varnothing$.
Since $r \geq 2$, we get
\[
\begin{split}
M_{21}
           &\lesssim \sum_{p> Q}\sum_{|q-p|\leq 2}\ls\lambda_q^{\frac{3}{r}}\| w_p\|_2\|w_q\|_{2}\sum_{Q<p'\leq p-2}\|\nabla v_{p'}\|_r\\
           &\lesssim \sum_{q> Q}\ls\| w_q\|_2^2\sum_{Q<p'\leq q}\lambda_{p'}\lambda^{\frac{3}{r}}_{q}\|v_{p'}\|_r\\
           &\lesssim c_{r,\delta}\nu\sum_{q> Q}\ls\| w_q\|_2^2\sum_{Q<p'\leq q}\lambda_{p'}^{1-\sigma-\frac{3}{r}}\lambda^{\frac{3}{r}}_q\lb^{1+\sigma}\\
           &\lesssim c_{r,\delta}\nu\sum_{q> Q}\l_q^{2+2s}\| w_q\|_2^2\sum_{Q<p'\leq q}\lambda_{p'}^{2-\frac{3}{r}}\l_q^{\frac{3}{r}-2}\\
           &\lesssim  c_{r,\delta}\nu \sum_{q> Q}\l_q^{2+2s}\| w_q\|_2^2\sum_{q\geq p}\left( L\lambda_{q-p} \right)^{\frac{3}{r}-2}
           \lesssim  c_{r,\delta}\nu \|w\|_{H^{s+1}}^2
\end{split}
\]
Note that the sum $\sum_{q \geq p}(L\l_{q-p})^{\frac{3}{r}-2}$ can be bounded uniformly. On the other hand, by definition of $\lb$ we have that $\lb^{-2}\|\nabla v_{\leq Q}\| < c_{r,\d} \nu$, and therefore, 
\[
\begin{split}
M_{22}
&\lesssim  \sum_{q> Q}\sum_{\substack{|q-p|\leq 2\\p>Q}}\ls \|w_p\|_2\|w_q\|_2\|\nabla v_{\leq Q}\|_\infty\\
&\lesssim  c_{r,\delta}\nu\sum_{q> Q}\lb^2\ls \|w_q\|_2^2\\
&\lesssim  c_{r,\delta}\nu \sum_{q> Q}\l_q^{2+2s} \|w_q\|_2^2 \lesssim  c_{r,\delta}\nu \|w\|_{H^{s+1}}^2. 
\end{split}
\]
To estimate $M_{3}$, we use integration by parts and obtain 
\[
\begin{split}
  M_{3}&=\sum_{q\geq -1}\ls\sum_{p\geq q-2} \left|\int_{\T^3}\Delta_q(\tilde w_p\cdot\nabla v_p)w_q\, dx\right|\\ 
  &\leq  \sum_{q>Q}\sum_{p\geq q-2}\ls\int_{\T^3}|\Delta_q(\tilde w_p \otimes v_{p}) \nabla w_q| \, dx\\
&\leq  \sum_{p> Q}\sum_{Q<q\leq p+2}\l_q^{1+2s}\|\tilde w_p\|_2\|w_q\|_{\frac{2r}{r-2}}\|v_p\|_r.
\end{split}
\]
By H\"older's inequality, Bernstein's inequality and definition of $\lb$ we have
\[
\begin{split}
M_{3}
&\lesssim  \sum_{p> Q}\|\tilde w_p\|_2\|v_{p}\|_r\sum_{Q<q\leq p+2}\lambda_q^{1+2s}\|w_q\|_{\frac{2r}{r-2}}\\
&\lesssim c_{r,\delta}\nu\sum_{p> Q}\lb^{1+\sigma}\l_p^{-\sigma-\frac{3}{r}}\|\tilde w_p\|_2\sum_{Q<q\leq p+2}\lambda_q^{1+2s+\frac{3}{r}}\|w_q\|_2\\
&\lesssim  c_{r,\delta}\nu\sum_{p> Q}\l_p^{1+s}\|\tilde w_p\|_2\sum_{Q<q\leq p+2}\lambda_q^{1+s}\|w_q\|_2\l_Q^{1+\sigma}\l_p^{-1-s-\sigma-\frac{3}{r}}\l_q^{s+\frac{3}{r}}\\
\end{split}
\]
then thanks to Jensen's inequalities,  
\[
\begin{split}
M_{3}&\lesssim  c_{r,\delta}\nu \sum_{p> Q}\l_p^{1+s}\|\tilde w_p\|_2\sum_{Q<q\leq p+2}\lambda_q^{1+s}\|w_q\|_2(L\l_{q-p})^{1+s+\sigma+\frac{3}{r}}\\
&\lesssim   c_{r,\delta}\nu \sum_{p> Q}\l_p^{2+2s}\|w_p\|_2^2+ c_{r,\delta}\nu \sum_{p>Q}\left(\sum_{Q<q\leq p+2}\lambda_q^{1+s}\|w_q\|_2(L\l_{q-p})^{1+s+\sigma+\frac{3}{r}}\right)^2\\
&\lesssim_{r,\d}  c_{r,\delta}\nu \sum_{p> Q}\l_p^{2+2s}\|w_p\|_2^2,
\end{split}
\]
Similar to the precise bound $\rf{precise_bound1}$, here we have that \\
\begin{equation}\label{precise_bound2}
    M_{3} \lesssim c_{r,\delta}\nu \|\nabla^{1+s} w\|_2^2\left(1+(1-2^{-(1+s+\sigma+\frac{3}{r})})^{-2}\right).
\end{equation}
\\
as this will determine our choice of $c_{r,\delta}$. Therefore, the estimate of $M_1$, $M_2$ and $M_3$ imply that
\begin{equation}\label{est-i1}
M\lesssim  c_{r,\delta} \nu \|\nabla^{1+s} w\|_2^2.
\end{equation}\\
The estimate of the term $N$ is essentially the same as before: with Bony's paraproduct formula we firstly write
\begin{equation}\notag
\begin{split}
N=&\int_{t_0}^t \sum_{q\geq -1}\ls\left|\int_{\T^3}\Delta_q(u\cdot\nabla w) w_q\, dx\right|\, d\tau\\
\leq&\sum_{q\geq -1}\sum_{|q-p|\leq 2}\ls\left|\int_{\T^3}\Delta_q(u_{\leq{p-2}}\cdot\nabla w_p)w_q\, dx\right|\\
&+\sum_{q\geq -1}\sum_{|q-p|\leq 2}\ls\left|\int_{\T^3}\Delta_q(u_{p}\cdot\nabla w_{\leq{p-2}})w_q\, dx\right|\\
&+\sum_{q\geq -1}\sum_{p\geq q-2}\sum_{|p-p'|\leq 1}\ls\left|\int_{\T^3}\Delta_q(u_p\cdot\nabla w_{p'})w_q\, dx\right|\\
=:& N_{1}+N_{2}+N_{3}.
\end{split}
\end{equation}
Secondly, realize that the term $N_1$ can be further decomposed by
\begin{equation}\notag
\begin{split}
N_{1}\leq& \sum_{q\geq -1}\sum_{|q-p|\leq 2}\ls\left|\int_{\T^3}[\Delta_q, u_{\leq{p-2}}\cdot\nabla] w_pw_q\, dx\right|\\
&+\sum_{q\geq -1}\ls\left|\int_{\T^3} u_{\leq q-2}\cdot\nabla w_q w_q\, dx\right|\\
&+\sum_{q\geq -1}\sum_{|q-p|\leq 2}\ls\left|\int_{\T    ^3}( u_{\leq{p-2}}- u_{\leq q-2})\cdot\nabla\Delta_qw_p w_q\, dx\right| \\
=&N_{11}+N_{12}+N_{13}.
\end{split}
\end{equation}
The term $N_{12}$ is obtained by the identity $\sum_{|p-q|\leq 2}\Delta_qw_p=w_q$. Furthermore, since $\divv u_{\leq q-2} = 0$, an integration by part implies that $N_{12} = 0$. The commutator in $N_{11}$ is given by
\[[\Delta_q, u_{\leq{p-2}}\cdot\nabla] w_p:=\Delta_q(u_{\leq{p-2}}\cdot\nabla w_p)-u_{\leq{p-2}}\cdot\nabla \Delta_qw_p.\]
 It can be deduced (see \cite{CD} for more details) that for any $1\leq r\leq \infty$,
\begin{equation}\notag
\|[\Delta_q, u_{\leq{p-2}}\cdot\nabla] w_p\|_r\\
\lesssim \|\nabla  u_{\leq p-2}\|_\infty\|w_p\|_r.
\end{equation}
Hence the estimate of $J_{11}$ is done as follows:
\[
\begin{split}
  N_{11} &\leq  \sum_{q>Q}\sum_{\substack{|q-p|\leq 2\\p>Q}}\ls\int_{\T^3}|[\Delta_q, u_{\leq{p-2}}\cdot\nabla] w_pw_q|\, dx\\
  &\leq  \sum_{q>Q}\sum_{\substack{|q-p|\leq 2\\p>Q}}\ls\|\nabla u_{(Q,p-2]}\|_\infty\|w_q\|_{2}\|w_p\|_2\\
  &+\sum_{q>Q}\sum_{\substack{|q-p|\leq 2\\p>Q}}\ls\|\nabla u_{\leq Q}\|_\infty\|w_p\|_2\|w_q\|_2\\
  &\equiv N_{111}+N_{112}.
\end{split}
\]
With Bernstein's inequality and definition of $\lb$ we get,
\[
\begin{split}
N_{111}
&\lesssim  \sum_{q>Q}\l_q^{2s}\|w_q\|_2^2\sum_{Q<p'\leq q}\l^{1+\frac{3}{r}}_{p'}\|u_{p'}\|_r\\
&\lesssim   c_{r,\delta}\nu\sum_{q>Q}\l_q^{2s}\|w_q\|_2^2\sum_{Q<p'\leq q}\lb^{1+\sigma}\l_{p'}^{1-\sigma}\\
&\lesssim    c_{r,\delta}\nu\sum_{q>Q}\l_q^{2+2s}\|w_q\|_2^2\sum_{Q<p'\leq q}\l_q^{1+\sigma}\l_{p'}^{1-\sigma}\l_q^{-2}\\
&\lesssim    c_{r,\delta}\nu \sum_{q>Q}\l_q^{2+2s}\|w_q\|_2^2,
\end{split}
\]
where we used $\sigma\geq -\haf$ . The term $N_{112}$ can be estimated similar as the term $M_{22}$, and so we obtain
\[
N_{112}
\lesssim    c_{r,\delta}\nu\lb^2\sum_{q>Q}\l_q^{2s}\|w_q\|_2^2
\lesssim    c_{r,\delta}\nu \sum_{q>Q}\l_q^{2+2s}\|w_q\|_2^2.
\]
Decomposing $N_{13}$ yields  
\[
\begin{split}
N_{13}
&\leq \sum_{q>Q}\sum_{\substack{|q-p|\leq 2\\p>Q}}\ls\int_{\T^3}|( u_{\leq{p-2}}- u_{\leq q-2})\cdot\nabla\Delta_qw_p w_q|\, dx\\
&\lesssim \sum_{q>Q}\l_q^{2s}\|u_{(q-4, Q]}\|_\infty\|\nabla w_q\|_2\|w_q\|_2+\sum_{q>Q}\sum_{\substack{q-4<p'\leq q\\p'>Q}}\l_q^{2s}\|u_{p'}\|_r\|\nabla w_q\|_{\frac{2r}{r-2}}\|w_q\|_2\\
&\lesssim \sum_{q>Q}\l_q^{1+2s}\|u_{(q-4, Q]}\|_\infty\|w_q\|_2^2+\sum_{q>Q}\sum_{\substack{q-4<p'\leq q\\p'>Q}}\l_q^{1+2s+\frac{3}{r}}\|u_{p'}\|_r\|w_q\|_2^2\\
&\equiv J_{131}+J_{132}.
\end{split}
\]
For $N_{131}$, realize that by definition of $\lb$ and Bernstein's inequality, one gets $\|u_{(q-4,Q]} \|_{\infty} < c_{r,\d} \nu \lb$, hence
\[
N_{131}\lesssim  c_{r,\delta}\nu \lb\sum_{q>Q}\l_q^{1+2s}\|w_q\|_2^2\lesssim  c_{r,\delta}\nu \sum_{q>Q}\l_q^{2+2s}\|w_q\|_2^2,
\]
Additionally, 
\[
\begin{split}
N_{132}& =\sum_{q>Q}\sum_{\substack{q-4\leq p'\leq q\\p'>Q}}\l_q^{1+2s+\frac{3}{r}}\|u_{p'}\|_r\|w_q\|_2^2\\
& \lesssim  c_{r,\delta}\nu\sum_{q>Q}\sum_{\substack{q-4\leq p'\leq q\\p'>Q}}\l_q^{1+2s+\frac{3}{r}}\lb^{1+\sigma}\l_{p'}^{-\sigma-\frac{3}{r}}\|w_q\|_2^2\\
& \lesssim  c_{r,\delta}\nu\sum_{q>Q}\l_q^{2+2s}\|w_q\|_2^2 \l_q^{-1-\sigma}\lb^{1+\sigma} \sum_{\substack{q-4\leq p'\leq q\\p'>Q}}\l_q^{\sigma+\frac{3}{r}}\l_{p'}^{-\sigma-\frac{3}{r}}\\
& \lesssim  c_{r,\delta}\nu\sum_{q>Q}\l_q^{2+2s}\|w_q\|_2^2(L\l_{Q-q})^{1+\sigma}\\
& \lesssim  c_{r,\delta}\nu\sum_{q>Q}\l_q^{2+2s}\|w_q\|_2^2,
\end{split}
\]
and notice that similar to the estimate for $M_{21}$, since $\sigma \geq -\haf$, we have a uniform bound for the term $\sup_{q> Q} (L\l_{Q-q})^{1+\sigma}$.\\\\
The term $N_2$ is estimated similarly as $I_1$:
\[
\begin{split}
  N_{2} &= \sum_{q>Q}\sum_{\substack{|q-p|\leq 2\\p>Q+2}}\ls\left|\int_{\T^3}\Delta_q(u_{p}\cdot\nabla w_{\leq{p-2}})w_q\, dx\right|\\
  &\leq \sum_{q>Q}\sum_{\substack{|q-p|\leq 2\\p>Q+2}}\ls\|u_p\|_r\|\nabla w_{(Q, p-2]}\|_{\frac{2r}{r-2}}\|w_q\|_2\\
&\lesssim   c_{r,\delta}\nu\sum_{q>Q}\sum_{\substack{|q-p|\leq 2\\p>Q+2}}\ls\lb^{1+\sigma}\l_p^{-\sigma-\frac{3}{r}}\|w_q\|_2\|\nabla w_{(Q, p-2]}\|_\frac{2r}{r-2}\\
&\lesssim   c_{r,\delta}\nu\sum_{q>Q}\lb^{1+\sigma}\l_q^{2s-\sigma-\frac{3}{r}}\|w_q\|_2\|\nabla w_{(Q, q]}\|_\frac{2r}{r-2}\\
&\lesssim   c_{r,\delta}\nu\sum_{q>Q}\lb^{1+\sigma}\l_q^{2s-\sigma-\frac{3}{r}}\|w_q\|_2\sum_{Q<p'\leq q}\l^{1+\frac{3}{r}}_{p'}\| w_{p'}\|_2\\       
&\lesssim   c_{r,\delta}\nu\sum_{q>Q}\l_q^{1+s}\|w_q\|_2\sum_{Q<p'\leq q}\l_{p'}^{1+s}\| w_{p'}\|_2\l_q^{s-\sigma-1-\frac{3}{r}}\l_{p'}^{-s+\frac{3}{r}}\lb^{1+\sigma}\\     
&\lesssim   c_{r,\delta}\nu\sum_{q>Q}\l_q^{1+s}\|w_q\|_2\left(\sum_{Q<p'\leq q}\l_{p'}^{1+s}\| w_{p'}\|_2(L\l_{q-p'})^{s-\sigma-1-\frac{3}{r}}\right)\\
&\lesssim   c_{r,\delta}\nu\sum_{q>Q}\l_q^{2+2s}\|w_q\|_2^2+ c_{r,\delta}\nu\sum_{q>Q}\left(\sum_{Q<p'\leq q}\l_{p'}^{1+s}\| w_{p'}\|_2(L\l_{q-p'})^{s-\sigma-1-\frac{3}{r}}\right)^2\\
&\lesssim_{r,\d}   c_{r,\delta}\nu\sum_{q>Q}\l_q^{2+2s}\|w_q\|_2^2,
\end{split}
\]
where we should highlight that the precise bound in the last term is 
\begin{equation}\label{precise_bound3}
N_{2} \lesssim c_{r,\delta}\nu \|\nabla^{1+s} w\|_2^2\left(1+(1-2^{s-\sigma-1-\frac{3}{r}})^{-2}\right).
\end{equation}
The term $N_3$ can be bounded by a similar manner as $M_3$, with the precise bound the same size as in $\rf{precise_bound2}$. 
Then the above estimates for $N$ lead to,
\begin{equation}\label{est-i2}
N \lesssim  c_{r,\delta} \nu \|\nabla^{1+s} w\|_2^2.
\end{equation}
Since the size of the precise bound $\eqref{precise_bound1}$ is larger than both $\eqref{precise_bound2}$ and \eqref{precise_bound3} asymptotically as $r \to \infty$ and $\d \to 0^+$, in view of \eqref{est-i1} and \eqref{est-i2}, we can choose a positive adimensional constant $c_0$ large enough such that 
\begin{equation}\label{est-i3}
M+N \leq C_{r,\d}  c_{r,\delta} \nu \|\nabla^{1+s} w\|_2^2,
\end{equation}
where 
\begin{equation}\label{def_C_rd}
C_{r,\d} := c_0 \Big ( 1+ (1-2^{s-\sigma-\frac{3}{r}})^{-2}\Big )
\end{equation}
Let 
\begin{equation}\label{def_c_rd}
c_{r,\delta}:=1/(2C_{r,\d})
\end{equation}
then $\eqref{est-i3}$ writes 
\begin{equation}\label{est-i4}
M+N \leq \haf \nu \|\nabla^{1+s} w\|_2^2,
\end{equation}
and hence from \eqref{eq-w2} we deduce that for all $t_0\leq t$,
\[
\begin{split}
\|w(t)\|_{H^s}^2 \leq & \|w(t_0)\|_{H^s}^2 - \nu \int_{t_0}^t \|\nabla^{1+s} w\|_2^2 \, d\tau\\
\leq & \|w(t_0)\|_{H^s}^2 - \nu \k^2_0 \int_{t_0}^t \|w\|^2_{H^s} \, d\tau,
\end{split}
\]
where we used the Poincar\'e's inequality and $\k_0=\frac{2\pi}{L}$.
Gr\"onwall's inequality then implies
\[
\|w(t)\|_{H^s}^2 \leq \|w(t_0)\|_{H^s}^2e^{-\nu\k_0^{2}(t-t_0)}, \qquad t\geq t_0.
\]\\
Taking the limit as $t \to \infty$ completes the proof.
\end{proof}

\begin{proof}[Proof of Corollary $\ref{main_Corollary}$]
    Let $u(t)$ and $v(t)$ be two Leray-Hopf solutions satisfying the condition in Theorem $\ref{main_Theorem}$. Let $s < 1$ and denote the difference again by $\om = u(t)-v(t)$, then from the energy inequality for $u(t)$ and $v(t)$,
\begin{align*}
    \|u(t)\|^2_2 &\leq \|u(t_0)\|^2_2 - 2\nu \int_{t_0}^t \|\nabla u(\tau)\|^2_2 d\tau +  \int_{t_0}^{t} (f,u)\, d\tau \\
    \|v(t)\|^2_2 &\leq \|v(t_0)\|^2_2 - 2\nu \int_{t_0}^t \|\nabla v(\tau)\|^2_2 d\tau + \int_{t_0}^{t} (f,v)\, d\tau
\end{align*}
and a simple estimation for the work done by $f$,
\begin{align*}
     \int_{t_0}^{t} (f,u)\, d\tau \lesssim \|f\|_{L^2_{loc}([0,\infty),H^{-1})}\|u\|_{L^2_{loc}([0,\infty),H^{1})} \lesssim 1 \\
     \int_{t_0}^{t} (f,v)\, d\tau \lesssim \|f\|_{L^2_{loc}([0,\infty),H^{-1})}\|v\|_{L^2_{loc}([0,\infty),H^{1})} \lesssim 1
\end{align*}
we can develop a bound for the dissipating term,
    \begin{align}
    2 \nu \int_{t_0}^t \|\nabla  \om (\tau) \|_2^2 d\tau &\leq 2\nu \bigg(\int_{t_0}^t \|\nabla u(\tau)\|^2_2 d\tau+\int_{t_0}^t \|\nabla v(\tau)\|^2_2 d\tau \bigg) \notag\\ 
    &\lesssim \|u(t_0) \|^2_2 + \|v(t_0)\|^2_2 \lesssim 1 \notag
    \end{align}
for any $t > t_0$. In particular, if we fix some $T > 0$, then for any $t > t_0$
\begin{equation}
\frac{2\nu}{T} \int_t^{t+T} \| \om (\tau) \|_{H^1}^2 d\tau \leq \frac{2\nu}{T} \int_{t_0}^{t+T} \| \om (\tau) \|_{H^1}^2 \lesssim 1 \label{bound_dissip_term}
\end{equation}
Recall from the definition $\eqref{def_Hs_norm}$ that we can write
\begin{align*}
    \|\om\|^2_{H^s} = \sum_{q=-1}^\infty \l^{2s}_q\|\om_q\|^2_2 = \sum_{q=-1}^\infty \l^{-\frac{5}{2p_1}}_q \|\om_q\|^{\frac{2}{p_1}}_2  \l^{\frac{5}{2p_1}+2s}_q\|\om_q\|^{\frac{2}{p_2}}_2
\end{align*}
where the tuple $(p_1,p_2)$ satisfies that
\begin{equation*}
    \frac{1}{p_1} + \frac{1}{p_2} = 1
\end{equation*}
By H{\"o}lder's inequality with the exponents
\begin{equation*}
    p_1 = \frac{9}{4-4s}, \ \ \ p_2 = \frac{9}{5+4s}
\end{equation*}
since 
\[
\frac{5}{2p_1}+2s = \frac{5}{9}(2-2s)+2s = \frac{10+8s}{9} = \frac{2}{p_2}
\]
we have 
\begin{align*}
    \|\om\|^2_{H^{s}} &\leq \bigg( \sum_{q=-1}^\infty \l^{-\frac{5}{2}}_q \|\om_q\|_2^2 \bigg)^\frac{2}{p_1} \bigg( \sum_{q=-1}^\infty \l^{2}_q \|\om_q\|_2^2 \bigg)^\frac{2}{p_2} \\
    &\leq \|\om\|^{\frac{2}{p_1}}_{H^{-\frac{5}{4}}} \|\om\|^{\frac{2}{p_2}}_{H^1}
\end{align*}
Therefore taking the average time integral on both sides of the above we obtain 
\begin{align*}
    \<\|u - v \|^2_{H^s} \>(t) &:= \frac{2\nu}{T} \int_t^{t+T} \|\om(\tau)\|_{H^s}^2 d\tau \leq \frac{2\nu}{T}\int_t^{t+T} \|\om\|^{\frac{2}{p_1}}_{H^{-\frac{5}{4}}} \|\om \|^{\frac{2}{p_2}}_{H^{1}} d\tau \\
    &\lesssim \sup_{t' \in [t,t+T]} \|\om(t')\|^{\frac{2}{p_1}}_{H^{-\frac{5}{4}}} \frac{2\nu}{T}\int_t^{t+T} \|\om \|^2_{H^{1}} d\tau \\
    &\lesssim \sup_{t' \in [t,t+T]}  \|\om(t')\|^{\frac{2}{p_1}}_{H^{-\frac{5}{4}}} 
\end{align*}
Observe that since $s < 1$, we have $\frac{2}{p_2} < 2$, leading to the second last inequality of the above. Also note that if $s = 1$, then $p_1 = +\infty$ and $p_2 = 1$
and the above inequality reduce to $\eqref{bound_dissip_term}$. Finally, from Theorem $\ref{main_Theorem}$ the right hand side of the last inequality goes to $0$ as $t \to \infty$. 
\end{proof}

\section{An upper bound for the average determining wavenumber} \label{sec:Kolmogorov}
In this section we develop a bound for the determining wavenumber $\lb_u$ in the time averaged sense. In view of the definition $\eqref{def_det_wave}$ and $\eqref{def_uni_det_wave}$, we can fix a tuple $(r,\d) \in \Ss$ and develop a uniform bound for the wavenumber $\lb^{r,\d}_u$. 
\begin{Lemma} \label{L:Lambda-main-estimate}
Let $(r,\d) \in \Ss$ and $\lb^{r,\d}_u$ be the wavenumber defined as in $\eqref{def_det_wave}$, if $\l_0 \leq \lb^{r,\d}_u < \infty$ we have that\\
\begin{equation} \label{eq:Lambda-main-estimate1}
c_{r,\delta} \nu(\lb^{r,\d}_u)^2 \lesssim  \|\nabla u_{\leq Q-1}\|_\infty + \sup_{p\geq Q}  (L\l_{p-Q})^{\sigma+\frac{3}{r}}(\lb^{r,\d}_u)^{1+\frac{3}{r}} \|u_p\|_r.
\end{equation}\\
and if $\lb^{r,\d}_u=\infty$, then
\begin{equation}\label{eq:Lambda-main-estimate2}
\sup_q \l_q^{\sigma+\frac{3}{r}} \|u_q\|_r = \infty.
\end{equation}\\
where $\sigma = \frac{\delta-1}{2}$, $Q \in \N$ is such that $\lb^{r,\d}_u = \l_Q$ and $c_{r,\d}$ is given in $\eqref{def_c_rd}$.
\end{Lemma}
\begin{Remark}\label{Rmk:Lambda-main-estimate}
The above lemma indicates that although the wavenumber $\lb^{r,\d}_u$ may not have a pointwise in time upper bound, a blow-up of the wavenumber implies a blow-up in higher mode of the solutions, i.e. a loss of regularity. Later, we will seek for a time-averaged bound for the wavenumber $\lb^{r,\d}_u$, uniformly on $(r,\d) \in \Ss$, and hence developing a time-averaged bound for the determining wavenumber $\lb_u$. 
\end{Remark}

\begin{proof}[Proof of lemma \ref{L:Lambda-main-estimate}]
Fix $(r,\d) \in \Ss$. Suppose that $\lb_u^{r,\d}=\infty$, by definition $\eqref{def_det_wave}$ we have that for $q\in \mathbb{N}$ either
\begin{equation} \label{1st-cond-in-L}
\sup_{p>q} (L\l_{p-q})^\sigma \l_{p}^{\frac{3}{r}}\l_{q}^{-1}\|u_p\|_{r} > c_{r,\delta}\nu ,
\end{equation}
or
\begin{equation} \label{2nd-cond-in-L}
\lambda_q^{-2}\|\nabla u_{\leq q}\|_{\infty} >  c_{r,\delta}\nu.
\end{equation}
\\
The first lower bound from the above implies that
\[
\limsup_{q\to \infty} \sup_{p>q} \l_q^{-\sigma-1} \l_{p}^{\sigma+\frac{3}{r}}\|u_{p}\|_{r} > c_{r,\delta}\nu.
\]
hence $\sup_q \l_q^{\sigma+\frac{3}{r}} \|u_q\|_r = \infty$, since $\sigma \geq - \haf$.\\
\\
The inequality \eqref{2nd-cond-in-L} leads to 
\[
\limsup_{q\to \infty} \lambda_q^{-2}\|\nabla u_{\leq q}\|_{\infty} > c_{r,\delta}\nu.
\]
On the other hand, since $\sigma \leq 1$, we obtain
\[
\begin{split}
\lambda_q^{-2}\|\nabla u_{\leq q}\|_{\infty} &\lesssim \l_q^{-2}\sum_{p\leq q} \l^{1+\frac{3}{r}}_p\|u_p\|_r\\
&=\l_q^{-\sigma-1}\sum_{p\leq q} (L\lambda_{q-p})^{\sigma-1}\l_p^{\sigma+\frac{3}{r}}\|u_p\|_r\\
&\lesssim q\l_q^{-\sigma-1} \sup_{p\leq q} \l_p^{\sigma+\frac{3}{r}}\|u_p\|_r.
\end{split}
\]
Hence,
\[
q^{-1}\l_q^{\sigma + 1} c_{r,\d} \nu \lesssim \sup_{p\leq q} \l_p^{\sigma+\frac{3}{r}}\|u_p\|_r 
\]
and again, sending $q \to \infty$ yields $\eqref{eq:Lambda-main-estimate2}$.
\\\\
Suppose now $\l_0 \leq \lb_u^{r,\d}(t) <\infty$,
then 
\begin{equation*}
\sup_{p>Q} (L\l_{p-Q})^\sigma \l_{p}^{\frac{3}{r}}\l_{Q}^{-1}\|u_p\|_{r} < c_{r,\delta}\nu ,
\end{equation*}
and 
\begin{equation*} 
\lambda_Q^{-2}\|\nabla u_{\leq Q}\|_{\infty} <  c_{r,\delta}\nu.
\end{equation*}
i.e. both conditions are satisfied in the definition $\eqref{def_det_wave}$ for $q = Q$. However, for $q = Q-1$, we have either
\begin{equation}\label{eq:alt1}
(L\l_{p-Q+1})^{\sigma}\l^{\frac{3}{r}}_{p}\l_{Q-1}^{-1}\|u_p\|_r > c_{r,\delta}\nu, \ \ \text{for some} \ \ p > Q -1,
\end{equation}
or
\begin{equation} \label{eq:alt2}
 \|\nabla u_{\leq Q-1}\|_\infty > c_{r,\delta} \nu\lambda_{Q-1}^2
\end{equation}
Rearranging the above inequalities we get either
\[
c_{r,\delta}\nu (\lb_u^{r,\d})^2 <  4 (L\l_{p-Q})^{\sigma+\frac{3}{r}}(\lb_u^{r,\d})^{1+\frac{3}{r}} \|u_p\|_r, \ \ \text{for some} \ \ p > Q-1, 
\]
or
\[
c_{r,\delta}\nu (\lb_u^{r,\d})^2 < 4 \|\nabla u_{\leq Q-1}\|_\infty.
\]\\
Adding the above and take supremum over $p > Q-1$ we obtain \eqref{eq:Lambda-main-estimate1}.
\end{proof}
As discussed in remark \ref{Rmk:Lambda-main-estimate}, consider now the average in time wavenumber
\[
\<\lb^{r,\d}_u\> (t):= \frac{1}{T}\int_t^{t+T} \lb_u^{r,\d}(\tau) \, d\tau,
\]
with some fixed $T>0$. Additionally, recall from $\eqref{eq:kdeps-inermit}$ the Kolmorgorov's dissipation number:
\[
\kappa_\mathrm{d} (t) = \left(\frac{\varepsilon }{\nu^3} \right)^{\frac{1}{d+1}}, \qquad  \varepsilon (t)= \nu\l_0^d\<\|\nabla u\|_2^2\> =  \frac{\nu \l_0^d}{T} \int_{t}^{t+T} \|\nabla u (\tau)\|_2^2 d\tau 
\]
with the intermittency dimension $d$ defined in $\eqref{eq:intermdef}$:
\[
d:= \sup\left\{s\in [0,3]: r \in [2,\infty], \left<\sum_{q\leq Q(t)}\l_q^{-1+\frac{6}{r}+s(1-\frac{2}{r})} \|u_q\|_r^2 \right> \lesssim  \l_0^{s(1-\frac{2}{r})}\left<\sum_{q \leq Q(t)}\l_q^{2} \|u_q\|_2^2 \right> \right\}
\]
The next theorem provide a uniform bound over the whole intermittency region $[0,3]$ on the average in time dissipation wavenumber $\<\lb_u\>(t)$.

\begin{Theorem}\label{Thm:avg_detwave_bound}
Let $d$ be the intermittency dimension and $\k_d$ be the Kolmogorov's dissipation wavenumber defined in $\rf{eq:intermdef}$ and $\rf{eq:kdeps-inermit}$ repsectively. The average determining wavenumber $\< \lb_u\> (t)$ has the following bound: 
\begin{enumerate}[label=\slshape(\roman*)]
  \item \label{avg_detwave_b1} For any $d \in [0,3]$, we have that
  \begin{equation*}
      \<\lb_u \> - \l_0 \lesssim \kappa_d \min \Big \{(\ln{(\l^{-1}_0\kappa_d)})^2, \frac{1}{d^2}\Big \},  
  \end{equation*}
  \item \label{avg_detwave_b2} Furthermore, if $d \in \big[ 0, (\ln{(\l^{-1}_0\kappa_d)})^{-1}\big] $, we have that \\
  \begin{equation*} 
      \<\lb_u \> - \l_0 \lesssim \kappa_d 
  \end{equation*}
\end{enumerate}
\end{Theorem}
\begin{Remark}
    Notice that the absolute constant appears in the bound does not depend on the value of $d$. Theorem $\ref{Thm:avg_detwave_bound}$ generalizes and improves the previous results on the bound of the average determining wavenumber in \cite{CD1} and \cite{CDK}. Part \ref{avg_detwave_b1} improves the bound on \cite{CD1} by removing the logarithmic correction for $d = 3$ and the (fixed)threshold constant $\d > 0$. Part \ref{avg_detwave_b2} agree with the bound in \cite{CDK} in the case of the extreme intermittency $d = 0$, and extend the intermittency range of the optimal bound from a single point to a small region, which depends on the size of $\k_d$. 
\end{Remark}

\begin{proof}
For any tuple $(r,\d) \in \Ss$, if $u$ is a solution, then by definition $\rf{def_uni_det_wave}$ we know that $\lb_u(t) \leq \lb^{r,\d}_u(t)$. Our goal is to develop an upper bound for $\lb^{r,\d}_u(t)$ with tuples $(r,\d) \in \Ss$ satisfying 
\begin{equation}\label{eq:additional_condition_rd}
d = \frac{\d}{1-\frac{2}{r}} \ \ \text{and}  \ \ \d < 3
\end{equation}
In another word, we set $r$ to be finite(If we let $r = \infty$, i.e. $d = \d$, then there will be a logarithmic correction when $d = 3$, see \cite{CD1}). In the following, for simplicity we will use $\lb(t)$ for $\lb_u^{r,\d}(t)$. Consider the case when $\lb(t)$ is finite, we can then use \eqref{eq:Lambda-main-estimate1}
in Lemma~\ref{L:Lambda-main-estimate} and Jensen's
inequality to obtain 
\[
\begin{split}
&(\lb(t) - \l_0)^{1+d(1-\frac{2}{r})}
 \lesssim  \frac{\lb^{d(1-\frac{2}{r})-3}}{(c_{r,\delta} \nu)^2}\left( \|\nabla u_{\leq Q-1}\|_\infty^2 + \sup_{q\geq Q}  (L\l_{q-Q})^{2(\sigma+\frac{3}{r})} \lb^{2(1+\frac{3}{r})}\|u_q\|_r^2\right)\\
&\lesssim  \frac{1}{(c_{r,\delta} \nu)^2}\left(\sum_{q\leq Q-1} \l_q^{\frac{-1+\frac{6}{r}+d(1-\frac{2}{r})}{2}}\|u_q\|_r (L\lambda_{Q-q})^{\frac{(d(1-\frac{2}{r})-3)}{2}} \right)^2
+ \frac{\lb^{-1+\frac{6}{r}+d(1-\frac{2}{r})}}{(c_{r,\delta}\nu)^2}\sup_{q\geq Q}  (L\l_{q-Q})^{2(\sigma+\frac{3}{r})} \|u_q\|_r^2\\
&\lesssim \frac{1}{(c_{r,\delta} \nu)^2} \sum_{q\leq Q-1} \l_q^{-1+\frac{6}{r}+d(1-\frac{2}{r})}\|u_q\|_r^2 + \frac{1}{(c_{r,\delta} \nu)^2}\sup_{q\geq Q}  (L\l_{q-Q})^{2\sigma-d(1-\frac{2}{r})+1} \l_q^{-1+\frac{6}{r}+d(1-\frac{2}{r})}\|u_q\|_r^2\\
&\lesssim \frac{1}{(c_{r,\delta}\nu)^2} \sum_{q=-1}^\infty \l_q^{-1+\frac{6}{r}+d(1-\frac{2}{r})}\|u_q\|_r^2.
\end{split}
\]
where we used the fact that $(r,\d)$ satisfies $\eqref{eq:additional_condition_rd}$ and hence $d(1-\frac{2}{r})-3 < 0$.
Note that if $\lb=\infty$, the inequality above is still true. This can be observed by $\eqref{eq:Lambda-main-estimate2}$ in lemma \ref{L:Lambda-main-estimate}
\[
\sum_{q} \l_q^{-1+\frac{6}{r}+d(1-\frac{2}{r})}\|u_q\|_r^2 = \sum_{q} \l_q^{\d-1+\frac{6}{r}}\|u_q\|_r^2 = \sum_{q} \l_q^{2\sigma+\frac{6}{r}}\|u_q\|_r^2 =\infty
\]
Now taking the time average integral and using the definition of the intermittency dimension $\eqref{eq:intermdef}$ yields
\[
\begin{split}
\<\lb\>-\l_0
& \lesssim \<(\lb - \l_0)^{1+d(1-\frac{2}{r})}\>^{\frac{1}{1+d(1-\frac{2}{r})}}\\
&\lesssim \left\<\frac{1}{(c_{r,\delta}\nu)^2} \sum_{q=-1}^\infty \l_q^{-1+\frac{6}{r}+d(1-\frac{2}{r})}\|u_q\|_r^2 
\right\>^{\frac{1}{1+d(1-\frac{2}{r})}}\\
&\lesssim \left\<\frac{\l_0^{d(1-\frac{2}{r})}}{(c_{r,\delta}\nu)^2} \sum_{q=-1}^\infty \l_q^{2}\|u_q\|_2^2 \right\>^{\frac{1}{1+d(1-\frac{2}{r})}}\\
&\lesssim \left\< \frac{\l_0^d\l^{-\frac{2d}{r}}_0}{c^2_{r,\delta}\nu^2}\sum_{q=-1}^\infty \l_q^{2}\|u_q\|_2^2 \right\>^{\frac{1}{1+d(1-\frac{2}{r})}}\\
&\lesssim  \left\< \frac{\l^{-\frac{2d}{r}}_0\nu \l_0^d}{c^2_{r,\delta}\nu^3 }\|\nabla u\|_2^2 \right\>^{\frac{1}{1+d(1-\frac{2}{r})}}\\
&=\l_0 \bigg(\frac{(\l^{-1}_0\kappa_d)^{d+1}}{c^2_{r,\delta}}\bigg)^{\frac{1}{1+d(1-\frac{2}{r})}}
\end{split}
\]
Recall that $c_{r,\delta}$ is explicitly written as $c_{r,\delta} = \frac{1}{2C_{r,\d}} = \frac{1}{2c_0}\bigg [ 1+\big(1-2^{s-\sigma-\frac{3}{r}}\big)^{-2} \bigg ]^{-1} $. It can then be shown that 
\bg \notag
c^{-1}_{r,\d} \lesssim \big (1-2^{-\frac{\d}{2}-\frac{3}{r}} \big)^{-2}
\ed 
for all $(r,\d)\in \Ss$ satisfying the condition $\eqref{eq:additional_condition_rd}$. We henceforward define the quantity
\bg \label{avg_detwave_bound}
M_d(r,\d) := \l_0 \big (1-2^{-\frac{\d}{2}-\frac{3}{r}} \big)^{-\frac{4}{1+d(1-\frac{2}{r})}} \big(\l_0^{-1}\kappa^{d+1}_d \big)^{\frac{1}{1+d(1-\frac{2}{r})}}
\ed
and so
\bg \notag
\<\lb\>-\l_0 \lesssim M_d(r,\d) 
\ed
for $(r,\d) \in \Ss$. We seek for $(r,\d)$ within the constraint of $\eqref{eq:additional_condition_rd}$ minimizing the size of $M_d(r,\d)$, and the key is to balance the size of $c^{-1}_{r,\d}$ and $\k^{\frac{d+1}{1+d(1-\frac{2}{r})}}_d$. To this end, consider the choice $r = 2\k^{\frac{d}{3}}_d$, and from which the values of $\d$ is determined (without loss of generality we assume that $\l_0 = 1$ for the time being)
\bg \notag
 \d = d \big( 1-\frac{2}{r} \big) = d \big( 1 - \k^{-\frac{d}{3}}_d \big )   
\ed
Denote the tuple $(r_m, \d_m) := \Big( 2\k^{\frac{d}{3}}_d, d \big( 1 - \k^{-\frac{d}{3}}_d \big )\Big )$, we obtain that 
\begin{align*}
    \<\lb\>-\l_0 &\lesssim M_d(r_m,\d_m) \lesssim \big (1-2^{-\frac{d}{2}-\frac{3-d}{r_m}} \big)^{-2} \k^{\frac{d+1}{1+\d_m}}_d  \\\\
    & = \Big (1-2^{-\frac{d}{2}-\frac{3-d}{r_m}} \Big)^{-2} \k^{\frac{d-\d_m}{1+\d_m}}_d \k_d
\end{align*}
We proceed by estimating the first two terms of the above. Observe that for any $d \in [0,3]$:
\[
\k^{\frac{d-\d_m}{1+\d_m}}_d \lesssim \k^{d\k^{-\frac{d}{3}}_d}_d \lesssim 1
\]
To see this, consider the quantity 
\[
Y(d) := \ln{\k^{d\k^{-\frac{d}{3}}_d}_d} = d\k^{-\frac{d}{3}}_d \ln{\k_d}
\]
and it suffices to compute the maximum of $Y$. The maximizing value $d = 3 (\ln{\k_d})^{-1}$ gives 
\bg \notag
Y \leq \frac{3}{e}
\ed
and so 
\[
\k^{d\k^{-\frac{d}{3}}_d}_d \leq e^\frac{3}{e}
\]
As for the first term, note that
\[
1-2^{-\frac{d}{2}-\frac{3-d}{r_m}} \gtrsim \frac{d}{2}\ln{2}
\] 
Therefore if $d > (\ln{\k_d})^{-1}$, \\
\[
\Big (1-2^{-\frac{d}{2}-\frac{3-d}{r_m}} \Big)^{-2} \lesssim \frac{1}{d^2} \lesssim (\ln{\k_d})^2
\]
which established the bound in \ref{avg_detwave_b1}. To prove \ref{avg_detwave_b2}, consider $0 \leq d \leq (\ln{\kappa_d})^{-1}$, then

\[
-\frac{d}{2}-\frac{3-d}{r_m} = -\frac{d}{2} - \frac{3-d}{2\k^{\frac{d}{3}}_d} \lesssim -\frac{3}{2e^{\frac{1}{3}}} 
\]
hence 
\[
\Big (1-2^{-\frac{d}{2}-\frac{3-d}{r_m}} \Big)^{-2} \lesssim \Big (1-2^{-\frac{d}{2}-\frac{(3-d)\k^{-\frac{d}{3}}_d}{2}} \Big)^{-2} \lesssim 1
\]\\
which proved \ref{avg_detwave_b2}.

\end{proof}



\end{document}